\documentclass[leqno,11pt]{amsart}

\usepackage{amsfonts}
\usepackage{amsmath}
\usepackage{amsfonts}
\usepackage{amssymb}
\usepackage{amscd}

\newtheorem{theorem}{Theorem}

\newtheorem{corollary}[theorem]{Corollary}

\newtheorem{definition}[theorem]{Definition}

\newtheorem*{example*}{Example}

\newtheorem{lemma}[theorem]{Lemma}

\newtheorem{prop}[theorem]{Proposition}
\newtheorem{remark}[theorem]{Remark}

\newcommand{\erre}{\mathbb{R}}
\newcommand{\p}{\mathbb{P}}

\newcommand{\enne}{\mathbb{N}}

\newcommand{\Hess}{\operatorname{Hess}}

\newcommand{\ricc}{\operatorname{Ric}}

\newcommand{\sgn}{\operatorname{sgn}}

\newcommand{\hess}{\operatorname{hess}}

\newcommand{\di}{\mathrm{d}} 
\newcommand{\R}{\operatorname{R}}

\newcommand{\ra}{\rightarrow}

\newcommand{\norm}[1]{{\left\|#1\right\|}}              
\newcommand{\pair}[1]{\left\langle#1\right\rangle}      

\renewcommand{\hat}[1]{\widehat{#1}}

\begin{document}

\title[Spacelike hyp. of constant $k$-mean curvature in GRW spacetimes]{Spacelike hypersurfaces of constant higher order mean curvature in generalized Robertson-Walker spacetimes}

\author[L. J. Al\'ias]{Luis J. Al\'ias}
\address{Departamento de Matem\'aticas, Universidad de Murcia, Campus de Espinardo, 30100 Espinardo, Murcia, Spain.}
\email{ljalias@um.es}
\thanks{This work was partially supported by MICINN project MTM2009-10418 and Fundaci\'{o}n S\'{e}neca
project 04540/GERM/06, Spain.
This research is a result of the activity developed within the framework of the Programme in Support of Excellence Groups of the Regi\'{o}n de Murcia, Spain, by Fundaci\'{o}n S\'{e}neca, Regional Agency
for Science and Technology (Regional Plan for Science and Technology 2007-2010).}
\author[D. Impera]{Debora Impera}
\address{Dipartimento di Matematica,
Universit\`a
degli studi di Milano, via Saldini 50, I-20133 Milano, Italy.}
\email{debora.impera@unimi.it}
\author[M. Rigoli]{Marco Rigoli}
\address{Dipartimento di Matematica,
Universit\`a
degli studi di Milano, via Saldini 50, I-20133 Milano, Italy.}
\email{marco.rigoli@unimi.it}
\date{August 13, 2011}


\begin{abstract}
In this paper we analyze the problem of uniqueness for spacelike hypersurfaces with constant higher order mean curvature in
generalized Robertson-Walker spacetimes. We consider first the case of compact spacelike hypersurfaces, completing some
previous results given in \cite{aliascolares}. We next extend these results to the complete noncompact case. In that case, our approach is based on the use of a generalized version of the Omori-Yau maximum principle for trace type differential operators, recently given in \cite{aliasimperarigoli}.
\end{abstract}
\maketitle
\section{Introduction}

Spacelike hypersurfaces in spacetimes are objects of increasing interest in recent years, both from physical and mathematical points of view. A basic question on this topic is the problem of uniqueness of spacelike hypersurfaces with constant mean curvature in certain spacetimes, and, more generally, that of spacelike hypersurfaces with constant higher order mean curvature.

In a recent paper, Al\'ias and Colares \cite{aliascolares} studied in depth the problem of uniqueness for compact spacelike hypersurfaces with constant
higher order mean curvature in spatially closed generalized Robertson-Walker spacetimes, that is, in generalized Robertson-Walker spacetimes having a compact
Riemannian factor. Their approach was based in the use of the so called Newton
transformations $P_k$ and their associated second order differential operators $L_k$ (see section \ref{prel}), as well as in the use of some general
Minkowski integral formulae for compact hypersurfaces.

In this paper, which is a natural continuation of \cite{aliascolares}, we go deeper into this study.
We consider first the case of compact spacelike hypersurfaces, extending the analysis and completing the results given in \cite{aliascolares}.
Specifically, instead of considering the assumption of the null convergence condition (NCC) as in \cite{aliascolares}, here we replace it by the condition
$(\log\rho)''\leq 0$ on the warping function $\rho$, which was introduced by Montiel in \cite{montiel2} and it is closely related to the timelike convergence
condition (TCC) and the NCC (see the beginning of Section \ref{compact} for further details). We next extend these results to the complete noncompact case. In that
case, our approach is based on the use of a generalized version of the Omori-Yau maximum principle for trace type differential operators which includes the
operators $L_k$ that has been recently introduced by the authors in \cite{aliasimperarigoli} for the study of hypersurfaces in Riemannian warped products.

The paper is organized as follows. After a preliminaries Section, where we fix notation and collect some basic results, we consider in Section \ref{compact} the problem of uniqueness for compact spacelike hypersurfaces in spatially closed generalized Robertson-Walker spacetimes. In particular, we obtain the following result which completes previous results in \cite{aliascolares}  (Theorem \ref{thmh2constcomplor}):
\begin{quote}
\begin{it}
Let $-I \times_{\rho} \p^n$ be a spatially closed generalized Robertson-Walker spacetime with warping function satisfying
$(\log \rho)''\leq 0$. The only compact spacelike hypersurfaces with $H_k$ constant, $2\leq k\leq n$, contained in a slab
$\Omega(t_1,t_2)$ on which $\rho'$ does not vanish are slices.
\end{it}
\end{quote}
In Section \ref{complete} we extend the previous results to the complete noncompact case with the aid of the generalized version of the Omori-Yau maximum principle given in \cite{aliasimperarigoli}. Among others, we obtain the following uniqueness result (Theorem \ref{thmGRW2}):
\begin{quote}
\begin{it}
Let $-I \times_{\rho} \p^n$ be a generalized Robertson-Walker spacetime whose warping function satisfies
$(\log \rho)''\leq 0$, with equality only at isolated points, and suppose that $\p^n$ has sectional curvature bounded from below.
Let $f:\Sigma^n \ra -I\times_{\rho} \p^n$ be a complete
spacelike hypersurface contained in a slab and assume that either
\begin{itemize}
\item[(i)] $H_2$ is a positive constant, or
\item[(ii)] $H_k$ is constant (with $k\geq 3$) and there exists an elliptic
point in $\Sigma$.
\end{itemize}
If  $\sup_{\Sigma}|H_1|< +\infty$, then $\Sigma$ is a slice.
\end{it}
\end{quote}
Finally, in Section \ref{further}, we give a number of further results extending to the complete case previous uniqueness results in \cite{aliascolares}.

\section{Preliminaries}\label{prel}

In what follows we consider a $n$-dimensional Riemannian manifold $\p^n$ and let $I$ be an open interval of the real line. We let $M^{n+1}:=-I \times_{\rho} \p^n$ to denote the Lorentzian warped product endowed with the Lorentzian metric
$$
\pair{,}=-\pi_I(dt^2)+\rho^2(\pi_I)\pi_{\p}(\pair{,}_{\p}).
$$
Following the terminology used in \cite{aliasromerosanchez} we will refer to $-I\times_{\rho}\p^n$ as a generalized Robertson-Walker spacetime. Observe that $\rho(t)\frac{\partial}{\partial t}$ is a closed conformal vector field on $M$ which determines a foliation $t\ra\p_t:=\{t\}\times\p$ of $M$ by complete totally umbilical spacelike hypersurfaces with constant mean curvature.\\
Consider a spacelike hypersurface $f:\Sigma^n \ra M^{n+1}$. In this case, since $T:=\frac{\partial}{\partial t}$ is a unitary timelike vector field globally defined on $M^{n+1}$, there exists a unique unitary timelike normal field $N$ globally defined on $\Sigma$ with the same orientation as $T$. Hence
$$
\Theta:=\pair{N,T}\leq -1 <0.
$$
We will refer to that normal field $N$ as the future-pointing Gauss map of the hypersurface.
We let $A: T\Sigma \rightarrow T\Sigma $ denote the second fundamental form of the immersion. Its eigenvalues $k_1,...,k_n$ are the principal curvatures of the hypersurface. Their elementary symmetric functions
$$S_k=\sum_{i_1<..<i_k}k_{i_1}\cdots k_{i_k},\qquad k=1,...,n,\qquad S_0=1,$$
define the $k$-mean curvatures of the immersion via the formula
$$
{n \choose k}H_k= (-1)^kS_k.
$$
Thus $H_1=-1/n\mathrm{Tr}(A)=H$ is the mean curvature and  $$n(n-1)H_2=\overline{S}-S+2\overline{\ricc}(N,N),$$ where $S$ and $\overline{S}$ are, respectively, the scalar curvature of $\Sigma$ and $M^{n+1}$ and $\overline{\ricc}$ is the Ricci tensor of the generalized Robertson-Walker spacetime. Even more, when $k$ is even, it follows from the Gauss equation that $H_k$ is a geometric quantity which is related to the intrinsic curvature of $\Sigma^n$.

The classical Newton transformations associated to the immersion are defined inductively by
$$
P_0=I,\qquad P_k={n \choose k}H_kI+AP_{k-1},
$$
for every $k=1,...,n$. It is not difficult to see that
\begin{enumerate}
\item[(a)] $\mathrm{Tr}(P_k)=c_k H_k$,\\
\item[(b)] $\mathrm{Tr}(AP_k)=-c_k H_{k+1}$,\\
\item[(c)] $\mathrm{Tr}(A^2P_k)={n\choose {k+1}}(nH_1H_{k+1}-(n-k-1)H_{k+2})$,
\end{enumerate}
where $c_k=(n-k){n \choose k}=(k+1){n\choose {k+1}}$. We refer the reader to \cite{aliasbrasilcolares} for further details (see also \cite{reilly} for other details about classical Newton tensors for hypersurfaces in Riemannian spaces).\\
Let $\nabla$ be the Levi-Civita connection of $\Sigma$. We define the second order linear differential operator $L_k:C^{\infty}(\Sigma) \ra C^{\infty}(\Sigma)$ associated to $P_k$ by
$$
L_k f=\mathrm{Tr}(P_k \circ \hess f),
$$
where
$$\pair{\hess f(X),Y}=\pair{\nabla_X \nabla f,Y}.
$$
It follows by the definition that the operator $L_k$ is elliptic if and only if $P_k$ is positive definite. Let us state two useful lemmas in which geometric conditions are given in order to guarantee the ellipticity of $L_k$ when $k \geq 1$ (Recall that $L_0=\Delta$ is always elliptic).
\begin{lemma}\label{lemmaelliptl1lor}
Let $\Sigma$ be a spacelike hypersurface immersed into a generalized Robertson-Walker spacetime. If $H_2 >0$ on $\Sigma$, then $L_1$ is an elliptic operator (for an appropriate choice of the Gauss map $N$).
\end{lemma}
For a proof of the Lemma see Lemma 3.2 in \cite{aliascolares}, where they proved it as a consequence of Lemma 3.10 in \cite{elbert}. The next Lemma is a consequence of Proposition 3.2 in \cite{barbosacolares} (see also Lemma 3.3 in \cite{aliascolares}).
\begin{lemma}\label{lemmaelliptlrlor}
Let $\Sigma$ be a spacelike hypersurface immersed into a generalized Robertson-Walker spacetime. If there exists an elliptic point of $\Sigma$, with respect to an appropriate choice of the Gauss map $N$, and $H_{k+1} >0$ on $\Sigma$, $2 \leq k \leq n-1$, then for all $1 \leq j \leq k$ the operator $L_j$ is elliptic.
\end{lemma}
Recall here that by an elliptic point in a spacelike hypersurface we mean a point of $\Sigma$ where all the principal curvatures are negative, with respect to an appropriate orientation.
In order to apply Lemma \ref{lemmaelliptlrlor} it is convenient to have some geometric conditions guaranteeing the existence of such a point.\\
The following technical lemma is a consequence of a more general result given in \cite{aliasbrasilcolares} and guarantees the existence of the elliptic point in the
compact case. Before stating it, recall from Proposition 3.2 (i) in \cite{aliasromerosanchez} that if a generalized Robertson-Walker spacetime
$-I\times_{\rho}\p^n$ admits a compact spacelike hypersurface, then the Riemannian factor $\p^n$ is necessarily compact. In that case, $-I\times_{\rho}\p^n$ is
said to be a spatially closed generalized Robertson-Walker spacetime.
\begin{lemma}[Lemma 5.3, \cite{aliascolares}]\label{lemma5.3}
Let $f:\Sigma^n\ra-I\times_{\rho}\p^n$ be a compact spacelike hypersurface immersed into a spatially closed generalized Robertson-Walker spacetime, and assume that $\rho'(h)$ does not vanish on $\Sigma$ (equivalently, $f(\Sigma)$ is contained in a slab $\Omega(t_1,t_2)$ on which $\rho'$ does not vanish).Then
\begin{itemize}
\item[(i)] if $\rho'(h)>0$ on $\Sigma$ (equivalently, $\rho'>0$ on $(t_1,t_2)$), then there exists an elliptic point on $\Sigma$ with respect to its future-pointing Gauss map.
\item[(ii)] if $\rho'(h)<0$ on $\Sigma$ (equivalently, $\rho'<0$ on $(t_1,t_2)$), then there exists an elliptic point on $\Sigma$ with respect to its past-pointing Gauss map.
\end{itemize}
\end{lemma}

For the proof of our main results we will make use of the following computations (see Section 4 in \cite{aliascolares}).
\begin{prop}\label{propsigmalor}
Let $f: \Sigma^n \ra M^{n+1}$ be a spacelike hypersurface. If
\begin{displaymath}
\sigma(t)=\int_{t_0}^t \rho(r) \di r,
\end{displaymath}
then
\begin{align}
L_k h=&-(\log \rho)'(h)(c_kH_k+\pair{P_k \nabla h,\nabla h})-\Theta c_k H_{k+1},\label{lrhlor}\\
L_k \sigma(h)=&-c_k(\rho'(h)H_k+\Theta\rho(h) H_{k+1}).\label{lrsigmalor}
\end{align}
\end{prop}
\begin{proof}
The gradient of $\pi_{\erre} \in C^{\infty}(M)$ is $\overline{\nabla}\pi_{\erre}=-T$, hence:
$$
\nabla h=(\overline{\nabla}\pi_{\erre})^T=-T-\Theta N.
$$
Moreover
$$
\nabla \sigma(h)=\rho(h)\nabla h=-\rho(h)T-\rho(h)\Theta N.
$$
Since $\rho(h)T$ is a non-vanishing closed conformal vector field on $M^{n+1}$ we have
$$
\overline{\nabla}_Z (\rho(t)T)=\rho'(t)Z,
$$
for every vector $Z$ tangent to $M^{n+1}$, where $\overline{\nabla}$ denotes the Levi-Civita connection of $M^{n+1}$.
Hence
$$
\overline{\nabla}_X\sigma(h)=-\rho'(h)X+\rho(h)\Theta AX-X(\rho(h)\Theta)N
$$
and
$$
\nabla_X \sigma(h)=(\overline{\nabla}_X\sigma(h))^T=-\rho'(h)X+\rho(h)\Theta AX.
$$
Then
\begin{align*}
L_k \sigma(h)=&\mathrm{Tr}(P_k \circ \hess(\sigma(h)))\\
=&-c_k\rho(h)\mathrm{Tr}(P_k)+\rho(h)\Theta\mathrm{Tr}(P_kA)\\
&-c_k(\rho'(h)H_k+\rho(h)\Theta H_{k+1}).
\end{align*}
Moreover
$$
\nabla_X \nabla h=-\frac{\rho'(h)}{\rho(h)}\pair{X,\nabla h}\nabla h+\frac{1}{\rho(h)}\nabla_X \sigma(h)$$
and therefore
\begin{align*}
L_k h=&-(\log \rho(h))'\pair{P_k \nabla h,\nabla h}+\frac{1}{\rho(h)}L_k\sigma(h)\\
=&-(\log \rho(h))'(\pair{P_k \nabla h,\nabla h}+c_kH_k)-c_k\Theta H_{k+1}.
\end{align*}
\end{proof}

\section{Uniqueness of compact spacelike hypersurfaces}\label{compact}
In \cite[Theorem 7]{montiel2} (see also \cite[Theorem 1]{aliasmontiel}) it was proved that the only compact spacelike hypersurfaces with constant mean curvature in a spatially closed generalized Robertson-Walker spacetime $-I \times_{\rho} \p^n$ whose warping function satisfies $(\log\rho)''\leq 0$ are the spacelike slices. Recall that a spacetime obeys the timelike convergence condition (TCC) if its Ricci curvature is nonnegative on timelike directions. It is not difficult to see that a generalized Robertson-Walker spacetime $-I \times_{\rho} \p^n$ obeys TCC if and only if
\begin{equation}\label{TCC1}
\ricc_\p\geq (n-1)\sup_I((\log\rho)''\rho^2)\pair{,}_\p,
\end{equation}
and
\begin{equation}\label{TCC2}
\rho''\leq0,
\end{equation}
where $\ricc_\p$ and $\pair{,}_\p$ are respectively the Ricci and metric tensors of the Riemannian manifold $\p$.

As observed in \cite{montiel2}, any of the two conditions above implies separately that spacelike slices are the only compact spacelike hypersurfaces in $-I \times_{\rho} \p^n$. In particular, the sole hypothesis \eqref{TCC2} suffices to guarantee uniqueness, without any other restriction on the curvature of $\p$. Even more, the more general condition $(\log\rho)''\leq0$ is sufficient to obtain the uniqueness. On the other hand, this uniqueness result was extended in \cite{aliascolares} to the case of compact spacelike hypersurfaces with $H_k$ constant under the assumption that the ambient generalized Robertson-Walker spacetime obeys the null convergence condition, which is nothing but \eqref{TCC1}, and that $\rho'$ does not vanish on the hypersurface. Related to this, here we obtain the following result.

\begin{theorem}\label{thmh2constcomplor}
Let $-I \times_{\rho} \p^n$ be a spatially closed generalized Robertson-Walker spacetime with warping function satisfying
$(\log \rho)''\leq 0$. The only compact spacelike hypersurfaces with $H_k$ constant, $2\leq k\leq n$, contained in a slab
$\Omega(t_1,t_2)$ on which $\rho'$ does not vanish are slices.
\end{theorem}
\begin{proof}
We may assume that $\rho'>0$ on $(t_1,t_2)$, so that $\rho'(h)>0$ on $\Sigma$. By Lemma \ref{lemma5.3} there exists an elliptic
point $p_0\in\Sigma$ with respect to the future-pointing Gauss map $N$. In particular, $H_k=H_k(p_0)$ is a positive
constant, $\Theta\leq -1$ on $\Sigma$, and by Lemma \ref{lemmaelliptlrlor} the operators $P_j$ are positive definite for all
$1\leq j\leq k-1$.

To make more transparent our reasoning, we first consider the case  $k=2$. We introduce the operator
$\mathcal{L}$ defined as
\[
\mathcal{L}=(n-1)(\log \rho)'(h)\Delta-\Theta L_1=\rm{Tr}(\mathcal{P}\circ\hess),
\]
where
\begin{displaymath}
\mathcal{P}=(n-1)(\log \rho)'(h)I-\Theta P_1.
\end{displaymath}
Since $\rho'(h)>0$, $-\Theta\geq 1>0$, and $P_1$ is positive definite, the operator $\mathcal{L}$ is elliptic. Moreover,
observe that
\[
\mathcal{L}(u)=\mathrm{div}(\mathcal{P}\nabla u)-\pair{\mathrm{div}\mathcal{P},\nabla u},
\]
where $\mathrm{div}\mathcal{P}=\mathrm{Tr}\nabla\mathcal{P}$. Hence by Theorem 3.1 in \cite{GT} $\mathcal{L}$ satisfies
the maximum principle.

It follows by Equation \eqref{lrsigmalor} that
\begin{equation}
\label{mathcall}
\mathcal{L}\sigma(h)=-c_1\rho(h)((\log\rho)'(h)^2-\Theta^2H_2).
\end{equation}
Compactness of $\Sigma$ implies the existence of points $p_{min}, p_{max}\in\Sigma$ such that
\begin{displaymath}
h(p_{min})=\min_{\Sigma}h=h_*,\qquad h(p_{max})=\max_{\Sigma}h=h^*.
\end{displaymath}
Hence $\nabla h(p_{min})=\nabla h(p_{max})=0$ and $\Theta(p_{min})=\Theta(p_{max})=-1$. Furthermore, since $\sigma$ is an increasing function,
$p_{min}$ and $p_{max}$ realize also the minimum and the maximum of $\sigma(h)$ on $\Sigma$, respectively.
Therefore, since $\mathcal{L}$ is an elliptic operator we have
\begin{align*}
\mathcal{L}\sigma(h)(p_{max})=-n(n-1)\rho(h^*)((\log\rho)'(h^*)^2-H_2)\leq 0,\\
\mathcal{L}\sigma(h)(p_{min})=-n(n-1)\rho(h_*)((\log\rho)'(h_*)^2-H_2)\geq 0.
\end{align*}
Hence
\begin{displaymath}
(\log\rho)'(h_*)^2\leq H_2\leq (\log\rho)'(h^*)^2,
\end{displaymath}
from which it follows that $(\log\rho)'(h)=H_2^{1/2}=\mathrm{constant}$, $(\log \rho)'$ being a positive non-increasing
function. Therefore, we have
\[
\mathcal{L}\sigma(h)=-n(n-1)\rho(h)((\log\rho)'(h)^2-\Theta^2H_2)=
-n(n-1)\rho(h)H_2(1-\Theta^2)\geq 0
\]
on $\Sigma$, and by compactness and the maximum principle we conclude that $\sigma(h)$, and hence $h$, is constant.

Let us consider now the case $k\geq 3$. Let $\mathcal{L}$ be the operator given by
\[
\mathcal{L}=\sum_{i=0}^{k-1}\frac{c_{k-1}}{c_i}(\log \rho)'(h)^{k-1-i}(-\Theta)^iL_i=\mathrm{Tr}(\mathcal{P}\circ \hess),
\]
where
\[
\mathcal{P}=\sum_{i=0}^{k-1}\frac{c_{k-1}}{c_i}(\log \rho)'(h)^{k-1-i}(-\Theta)^iP_i.
\]
Since $\rho'(h)>0$, $-\Theta\geq 1>0$, and $P_1,\ldots, P_{k-1}$ are all positive definite, the operator
$\mathcal{L}$ is elliptic and it satisfies the maximum principle.

We claim that
\begin{equation}
\label{mathcallrcompact}
\mathcal{L}\sigma(h)=-c_{k-1}\rho(h)\left((\log\rho)'(h)^{k}-(-\Theta)^{k}H_{k}\right).
\end{equation}
We can prove the claim by induction. We have already proved that this is true for $k=2$. Assuming that it is true for
$k-2$ and using \eqref{lrsigmalor} we get
\begin{align*}
\mathcal{L}\sigma(h)=&\frac{c_{k-1}}{c_{k-2}}(\log \rho)'(h)
\sum_{i=0}^{k-2}\frac{c_{k-2}}{c_i}(\log\rho)'(h)^{k-2-i}(-\Theta)^iL_i\sigma(h)\\
{}&+(-\Theta)^{k-1}L_{k-1}\sigma(h)\\
=&-c_{k-1}\rho(h)(\log\rho)'(h)^{k}+c_{k-1}\rho(h)(\log\rho)'(h)(-\Theta)^{k-1}H_{k-1}\\
{}&-c_{k-1}\rho(h)\left((\log\rho)'(h)(-\Theta)^{k-1}H_{k-1}-(-\Theta)^{k}H_{k}\right)\\
=& -c_{k-1}\rho(h)\left((\log\rho)'(h)^{k}-(-\Theta)^{k}H_{k}\right).
\end{align*}
The rest of the proof follows as in the case $k=2$, using \eqref{mathcallrcompact} instead of \eqref{mathcall}.

\end{proof}

As an application of Theorem \ref{thmh2constcomplor} we have the following
\begin{corollary}\label{coro6}
Let $-I \times_{\rho} \p^n$ be a spatially closed generalized Robertson-Walker spacetime with warping function satisfying
$(\log \rho)''\leq 0$. The only compact spacelike hypersurfaces with nonvanishing mean curvature and $H_k$ constant, $2\leq k\leq n$, are slices.
\end{corollary}
\begin{proof}
We may choose the orientation so that $H_1>0$. In this case, since $\Theta$ never vanishes, it can be either positive or negative with respect to this orientation. \\
Let us assume first that $\Theta<0$. Since $\Sigma$ is compact there exist points $p_{max}$ and $p_{min}$ at which the height function $h$ attains its maximum and minimum values respectively. In particular, $\nabla h(p_{max})=0$ and $\Theta(p_{max})=-1$. Moreover, setting $h^*:=h(p_{max})$ and using Proposition \ref{propsigmalor}
$$
0\geq\Delta h(p_{max})=-n(\log\rho)'(h^*)+nH_1(p_{max})> -n(\log\rho)'(h^*).$$
Observing now that $-\log\rho$ is a convex function we get
$$
(\log\rho)'(h)\geq(\log\rho)'(h^*)> 0
$$
and hence $\rho'(h)>0$.\\
On the other hand, let us consider the case of $\Theta>0$ with respect to the chosen orientation. As already said, we can find a point $p_{min}$ were the height function attains its minimum. In this case $\nabla h(p_{min})=0$, $\Theta(p_{min})=1$ and
$$
0\leq\Delta h(p_{min})=-n(\log\rho)'(h_*)-nH_1(p_{min})< -n(\log\rho)'(h_*),$$
where $h_*=h(p_{min})$. Reasoning as above we can see that $\rho'(h)<0$.\\
In any case, we conclude that $\Sigma$ is contained in a slab on which $\rho'$ does not vanish and the result follows from Theorem \ref{thmh2constcomplor}.
\end{proof}
\begin{corollary}\label{coro7}
Let $-I \times_{\rho} \p^n$ be a spatially closed  generalized Robertson-Walker spacetime whose warping function satisfies
$(\log \rho)''\leq 0$.
The only compact spacelike hypersurfaces satisfying either
\begin{itemize}
\item[(i)] $H_2$ is a positive constant, or
\item[(ii)] $H_k$ is constant (with $k\geq 3$) and there exists an elliptic
point in $\Sigma$,
\end{itemize}
are slices.
\end{corollary}
For the proof of this corollary, observe that in case $(i)$, by the basic inequality $H_1^2\geq H_2>0$, it follows that the mean curvature does not vanish. On the other hand, in case $(ii)$, we assume that there exists a point $p_0\in\Sigma$ where all the principal curvatures are negative. Therefore, the constant $H_k=H_k(p_0)$ is positive and, using the Garding inequalities \cite{Ga} we get that
$$
H_1\geq H_2^{1/2}\geq\cdots\geq H_k^{1/k}>0
$$
on $\Sigma$. In particular, $H_1>0$ and the conclusion follows by Corollary \ref{coro6}.

\section{Uniqueness of complete spacelike hypersurfaces}\label{complete}

We will now extend the previous theorems to the complete noncompact case. To do that we will use a generalization of the Omori-Yau maximum principle for trace type differential operators.

Let $\Sigma$ be a Riemannian manifold and let $L=\mathrm{Tr}(P\circ\hess)$ be a semi-elliptic operator, where $P:T\Sigma\ra T\Sigma$ is a positive semi-definite symmetric tensor. Following the terminology introduced in \cite{pirise}, we say that the Omori-Yau maximum principle holds on $\Sigma$ for the operator $L$ if, for any function $u \in C^{2}(\Sigma)$ with
$u^*=\sup_{\Sigma}u < +\infty$, there exists a sequence $\left\{p_{j}\right\}_{j\in \enne} \subset \Sigma$ with the properties
\[
\mathrm{(i)}\   u(p_j)>u^*-\frac{1}{j},\ \mathrm{(ii)}\   \norm{\nabla u(p_j)}<\frac{1}{j}, \ \mathrm{(iii)}\   Lu(p_j)< \frac{1}{j}
\]
for every $j\in\enne$. Equivalently, for any function $u \in C^{2}(\Sigma)$ with
$u_*=\inf_{\Sigma}u > -\infty$, there exists a sequence $\left\{p_{j}\right\}_{j\in \enne} \subset \Sigma$ with the properties
\[
\mathrm{(i)}\   u(p_j)<u_*+\frac{1}{j},\ \mathrm{(ii)}\   \norm{\nabla u(p_j)}<\frac{1}{j}, \ \mathrm{(iii)}\   Lu(p_j)> -\frac{1}{j}
\]
for every $j\in\enne$.

In \cite[Theorem 1]{aliasimperarigoli} the authors have recently proved the following version of a generalized Omori-Yau maximum principle for trace type
differential operators.
\begin{theorem}[Theorem 1 in \cite{aliasimperarigoli}]
\label{maxprinc}
Let $(\Sigma,\pair{,})$ be a Riemannian manifold, and let $L=\mathrm{Tr}(P \circ \hess)$ be a semi-elliptic operator, where
$P:T\Sigma\rightarrow T\Sigma$ is a positive semi-definite symmetric tensor satisfying $\sup_\Sigma\mathrm{Tr}P<+\infty$.
Assume the existence of a non-negative $C^2$ function $\gamma$ with the properties
\begin{eqnarray}
\label{gamma1} &\gamma(p) \ra +\infty \qquad &\text{as } p \ra \infty,\\
\label{gamma2} &\exists A>0 \qquad &\text{such that }\norm{\nabla \gamma}\leq A\sqrt{\gamma}\qquad \text{off a compact set,}\\
\label{gamma3} &\exists B>0 \qquad &\text{such that } L \gamma \leq B\sqrt{\gamma\, G(\sqrt{\gamma})}\qquad \text{off a compact set,}
\end{eqnarray}
where $G$ is a smooth function on $[0,+\infty)$ such that:
\begin{equation}\label{condG}
\begin{array}{ll}
\mathrm{(i)}\  G(0)>0, & \mathrm{(ii)}\  G'(t)\geq 0 \qquad \text{on } [0,+\infty),\\
\mathrm{(iii)}\  1/\sqrt{G(t)}\not \in L^1(+\infty),& \mathrm{(iv)}\  \limsup_{t \ra \infty} \frac{t G(\sqrt{t})}{G(t)}<+\infty.
\end{array}
\end{equation}
Then, the Omori-Yau maximum principle holds on $\Sigma$ for the operator $L$.
\end{theorem}

As a consequence of Theorem \ref{maxprinc}, we have the following result, which will be essential for the proof of our main results
(see also Corollary 3 in \cite{aliasimperarigoli} for a more general version in terms of a (generally non-constant) lower bound for the radial sectional curvature).
\begin{lemma}
\label{lemmaOY}
Let $(\Sigma,\pair{,})$ be a complete, noncompact Riemannian manifold with sectional curvature bounded from below. Then, the Omori-Yau maximum principle holds on
$\Sigma$ for any semi-elliptic operator $L=\mathrm{Tr}(P \circ \hess)$ with $\sup_\Sigma\mathrm{Tr}P<+\infty$.
\end{lemma}
\begin{proof}
Let $o\in\Sigma$ be a fixed reference point, denote with $r(p)$ the
distance function from $o$ and set $\gamma(p)=r(p)^2$. Then $\gamma$ satisfies assumptions \eqref{gamma1} and \eqref{gamma2} of
Theorem \ref{maxprinc}. Furthermore, $\gamma$ is smooth within the cut locus of $o$.
Assume that the sectional curvature of $\Sigma$ is bounded from below by a constant $c$. Since $\Sigma$ is assumed to be complete and
noncompact, then $c\leq 0$. Then, by the Hessian comparison theorem within the cut locus of $o$, one has
\begin{equation}
\label{luis.5}
\Hess{r}(p)(v,v)\leq\psi_c(r(p))(\norm{v}^2-\pair{\nabla r(p),v}^2)
\end{equation}
for every $v\in T_p\Sigma$, where $\psi_c(t)$ is given by
\[
\psi_c(t)
=\left\{\begin{array}{lll}
1/t & \mathrm{if} & c=0,\\
\sqrt{-c}\coth(\sqrt{-c}\, t) & \mathrm{if}  & c<0.
\end{array}\right.
\]
Since  $\Hess{\gamma}=2r\Hess{r}+2dr\otimes dr$, we obtain from here that
\begin{eqnarray}
\label{luis.20}
\nonumber \Hess{\gamma} & \leq & 2\sqrt{\gamma}\,\psi_c(\sqrt{\gamma})\pair{,}+2(1-\sqrt{\gamma}\psi_c(\sqrt{\gamma})dr\otimes dr \\
{} & \leq & 2\sqrt{\gamma}\,\psi_c(\sqrt{\gamma})\pair{,}
\end{eqnarray}
for $\gamma$ sufficiently large, since $1-t\psi_c(t)\leq 0$ if $t\gg 1$.
Then, using the fact that $P$ is positive semi-definite we get
\[
L\gamma\leq 2\mathrm{Tr}P \sqrt{\gamma}\,\psi_c(\sqrt{\gamma})
\]
for $\gamma$ sufficiently large. Since $\sup_\Sigma\mathrm{Tr}P<+\infty$ and $\lim_{t\rightarrow+\infty}\psi_c(t)=\sqrt{-c}$, then we conclude that
\[
L\gamma\leq B \sqrt{\gamma\,G(\sqrt{\gamma})}
\]
for a positive constant $B$ and $\gamma$ sufficiently large, where $G(t)$ is given (for instance) by $G(t)=t^2$ with $t\gg 1$. Therefore, by
Theorem \ref{maxprinc} we know that the Omori-Yau maximum principle holds on $\Sigma$ for $L$.
\end{proof}

Lemma \ref{lemmaOY} has the following application in our situation.
\begin{corollary}
\label{OYsigma}
Let $-I \times_{\rho} \p^n$ be a generalized Robertson-Walker spacetime with warping function satisfying $(\log\rho)''\leq0$ and Riemannian fiber $\p^n$ having sectional curvature bounded from below. Let $f:\Sigma^n \ra -I \times_{\rho} \p^n$ be a complete spacelike hypersurface contained in a slab with
$\sup_\Sigma\norm{A}^2<+\infty$.
Then the sectional curvature of $\Sigma$ is bounded from below and the Omori-Yau maximum principle holds on $\Sigma$ for every semi-elliptic operator
$L=\mathrm{Tr}(P \circ \hess)$ with $\sup_\Sigma\mathrm{Tr}P<+\infty$.
\end{corollary}

\begin{remark}\label{remarkmarco}
\rm{
From the equality
$$
\norm{A}^2=n^2H_1^2-n(n-1)H_2
$$
it follows that under the assumption $\inf_\Sigma H_2>-\infty$ the condition $\sup_\Sigma \norm{A}^2<+\infty$ is
equivalent to $\sup_\Sigma|H_1|<+\infty$.
}
\end{remark}
\begin{proof}[Proof of Corollary \ref{OYsigma}]
Recall the Gauss equation
$$
\R(X,Y)Z=(\overline{\R}(X,Y)Z)^T-\pair{AX,Z}AY+\pair{AY,Z}AX,
$$
for all vector fields $X,\ Y,\ Z$ tangent to $\Sigma$, where $\R$ and $\overline{\R}$ are the curvature tensors of $\Sigma^n$ and $-I\times_\rho\p$, respectively.
Then, if $\{X, Y\}$ is an orthonormal basis for an arbitrary 2-plane tangent to $\Sigma$, the sectional curvature in $\Sigma$ of that 2-plane is given by
\begin{align}
\label{luis.7}
\nonumber K_{\Sigma}(X,Y)=&\overline{K}(X,Y)-\pair{AX,X}\pair{AY,Y}+\pair{AX,Y}^2\\
\geq & \overline{K}(X,Y)-\norm{AX}\norm{AY}\\
\nonumber \geq & \overline{K}(X,Y)-\|A\|^2,
\end{align}
where $\overline{K}(X,Y)$ denotes the sectional curvature in $-I\times_\rho\p$ of the plane spanned by $\{X, Y\}$. Observe that
the last inequality follows from the fact that
$$
\norm{AX}^2\leq \text{Tr}(A^2)\norm{X}^2=\|A\|^2
$$
for every unit vector $X$ tangent to $\Sigma$. Since we are assuming that $\sup_{\Sigma}\|A\|^2<+\infty$, it suffices to have
$\overline{K}(X,Y)$ bounded from below.
A direct computation using the general relationship between the curvature tensor of a warped product and the curvature tensor of its base and its fiber, as well as the derivatives of the warping function (see for instance Proposition 42 in \cite{oneill}) implies that
\begin{align*}
\overline{\R}(U,V)W=&\R_{\p}({\pi_{\p}}_*U,{\pi_{\p}}_*V){\pi_{\p}}_*W+((\log\rho)')^2(\pi_{\erre})(\pair{U,W}V-\pair{V,W}U)\\
&+(\log\rho)''(\pi_{\erre})\pair{W,T}(\pair{V,T}U-\pair{U,T}V)\\&-(\log\rho)''(\pi_{\erre})(\pair{U,W}\pair{V,T}-\pair{V,W}\pair{U,T})T.
\end{align*}
for every $U,V,W\in TM$, where $T=\frac{\partial}{\partial_t}$ and we are using the notation $U^*$ to denote ${\pi_{\p}}_*U$ for an arbitrary
$U\in TM$. Then, for the orthonormal basis $\{X,Y\}$ we find that
\begin{eqnarray}
\label{luis.8}
\nonumber \overline{K}(X,Y) & = & \frac{1}{\rho^2(h)} K_{\p}(X^*,Y^*)
\norm{X^*\wedge Y^*}^2\\
{} & {} &+((\log\rho'))^2(h)-(\log\rho)''(h)(\pair{X,\nabla h}^2+\pair{Y,\nabla h}^2)\\
\nonumber {} & \geq & \frac{1}{\rho^2(h)} K_{\p}(X^*,Y^*)\norm{X^*\wedge Y^*}^2.
\end{eqnarray}
On the other hand,
\begin{eqnarray*}
\norm{X^*\wedge Y^*}^2 & = &
\norm{X^*}^2\norm{Y^*}^2-\pair{X^*,Y^*}^2\\
{} & = & 1-\pair{X,T}^2-\pair{Y,T}^2\leq 1.
\end{eqnarray*}
Therefore, if $K_{\p}\geq c$ for some constant $c$, we deduce
\begin{equation}
\label{luis.9}
\frac{1}{\rho^2(h)} K_{\p}(X^*,Y^*)\norm{X^*\wedge Y^*}^2
\geq -\frac{|c|}{\rho^2(h)}.
\end{equation}
Finally, since $h$ is a bounded function, we conclude from \eqref{luis.7}, \eqref{luis.8} and \eqref{luis.9} that
the sectional curvature $K(X,Y)$ is bounded from below by an absolute constant.
\end{proof}

Now we are ready to state the main results of this section.
\begin{theorem}\label{thmconsth2h1}
Let $-I \times_{\rho} \p^n$ be a generalized Robertson-Walker spacetime whose warping function satisfies
$(\log \rho)''\leq 0$, with equality only at isolated points, and suppose that $\p^n$ has sectional curvature bounded from below.
Let $f:\Sigma^n \ra -I\times_{\rho} \p^n$ be a complete
spacelike hypersurface contained in a slab with $H_k>0$, for some $2\leq k\leq n$, and $\frac{H_{i+1}}{H_i}=\mathrm{constant}$ for some
$1\leq i\leq k-1$. Assume that  $\sup_{\Sigma}|H_1|< +\infty$ and, for $k\geq 3$, that there exists an
elliptic point in $\Sigma$. Then, $\Sigma$ is a slice.
\end{theorem}
\begin{proof}
First we consider the case $k=2$. From the basic inequality $H_1^2\geq H_2>0$, it follows that we can orient the hypersurface
so that $H_1>0$ on $\Sigma$. We define the operator $\hat{L}_1=\mathrm{Tr}(\hat{P}_1\circ\hess)$ with
$\hat{P}_1=\frac{1}{H_1}P_1$. Note that $\mathrm{Tr}(\hat{P}_1)=c_1$ and therefore, by Corollary \ref{OYsigma} and Remark \ref{remarkmarco}, the Omori-Yau
maximum principle holds on $\Sigma$ for the operator $\hat{L}_1$. We let $\{p_j\}$ and $\{q_j\}$ be two sequences
such that
\begin{align*}
(i) \quad & \lim_{j\ra +\infty} \sigma(h(p_j))=\sup_\Sigma\sigma(h),\\
(ii) \quad & \norm{\nabla\sigma(h)(p_j)}=\rho(h(p_j))\norm{\nabla h(p_j)}<\frac{1}{j},\\
(iii) \quad & \hat{L}_1 \sigma(h)(p_j)<\frac{1}{j},
\end{align*}
and
\begin{align*}
(i) \quad & \lim_{j\ra +\infty} \sigma(h(q_j))=\inf_\Sigma\sigma(h),\\
(ii) \quad & \norm{\nabla\sigma(h)(q_j)}=\rho(h(q_j))\norm{\nabla h(q_j)}<\frac{1}{j},\\
(iii) \quad & \hat{L}_1 \sigma(h)(q_j)>-\frac{1}{j},
\end{align*}
Observe that condition (i) implies that $\lim_{j\ra +\infty}h(p_j)=h^*=\sup_\Sigma h$
and $\lim_{j\ra +\infty}h(q_j)=h_*=\inf_\Sigma h$, because $\sigma(t)$ is strictly increasing. Thus by
condition (ii) we also have $\lim_{j\ra +\infty}\norm{\nabla h(p_j)}=\lim_{j\ra +\infty}\norm{\nabla h(q_j)}=0$, and
$\lim_{j\ra +\infty}\Theta(p_j)=\lim_{j\ra +\infty}\Theta(q_j)=\mathrm{sign}\Theta$.
Therefore, using that
\[
\hat{L}_1\sigma(h)=-c_1\left(\rho'(h)+\Theta\rho(h)\frac{H_2}{H_1}\right)
\]
we get
\[
\frac{1}{j}>\hat{L}_1 \sigma(h)(p_j)=-c_1\left(\rho'(h(p_j))+\Theta(p_j)\rho(h(p_j))\frac{H_2}{H_1}\right)
\]
and
\[
-\frac{1}{j}<\hat{L}_1 \sigma(h)(q_j)=-c_1\left(\rho'(h(q_j))+\Theta(q_j)\rho(h(q_j))\frac{H_2}{H_1}\right).
\]
Making $j\rightarrow+\infty$ in these inequalities we find
\[
(\log\rho)'(h_*)\leq -\mathrm{sign}\Theta\frac{H_2}{H_1}\leq(\log\rho)'(h^*).
\]
Using the assumption $(\log\rho)''\leq 0$ with equality only at isolated points, we conclude from here that $h$ is constant.

For the general case $k\geq 3$, first observe that the existence of an elliptic point and $H_k>0$ implies that
$H_i>0$ and the operators $P_i$ are positive definite for all $1\leq i\leq k-1$. Choose $i$ as in the statement of the theorem,
so that $H_{i+1}/H_i$ is constant and consider the operator $\hat{L}_i=\mathrm{Tr}(\hat{P}_i\circ\hess)$ with
$\hat{P}_i=\frac{1}{H_i}P_i$. Note that $\mathrm{Tr}(\hat{P}_i)=c_i$ and therefore, by Corollary \ref{OYsigma} and Remark \ref{remarkmarco}, the
Omori-Yau maximum principle holds on $\Sigma$ for the operator $\hat{L}_i$. We conclude then as in the case $k=2$ with the aid
of the equation
\[
\hat{L}_i\sigma(h)=-c_i\left(\rho'(h)+\Theta\rho(h)\frac{H_{i+1}}{H_i}\right).
\]
\end{proof}
The next result extends to the complete case Corollary \ref{coro7}.
\begin{theorem}\label{thmGRW2}
Let $-I \times_{\rho} \p^n$ be a generalized Robertson-Walker spacetime whose warping function satisfies
$(\log \rho)''\leq 0$, with equality only at isolated points, and suppose that $\p^n$ has sectional curvature bounded from below.
Let $f:\Sigma^n \ra -I\times_{\rho} \p^n$ be a complete
spacelike hypersurface contained in a slab and assume that either
\begin{itemize}
\item[(i)] $H_2$ is a positive constant, or
\item[(ii)] $H_k$ is constant (with $k\geq 3$) and there exists an elliptic
point in $\Sigma$.
\end{itemize}
If  $\sup_{\Sigma}|H_1|< +\infty$, then $\Sigma$ is a slice.
\end{theorem}
\begin{proof}
Consider first the case $k=2$. Since $H_2 >0$ it follows by Lemma \ref{lemmaelliptl1lor} that the operator $L_1$ is elliptic with respect to the orientation for which $H_1>0$ (see the proof of Lemma 3.2 in \cite{aliascolares} for the details). Assume first that $\Theta<0$ with respect to this orientation and let us show that $\rho'(h)>0$. To do that, we apply the Omori-Yau maximum principle to the Laplacian to assure the existence of a sequence $\{p_j\}$ with the following properties
\begin{align*}
(i) \quad & \lim_{j\ra +\infty} h(p_j)=\inf_\Sigma h=h^*,\\
(ii) \quad & \norm{\nabla h(p_j)}<\frac{1}{j},\\
(iii) \quad & \Delta h(p_j)<\frac{1}{j}.
\end{align*}
Therefore, making $j\ra+\infty$ in the following inequality
$$
\frac1j>\Delta h(p_j)=-(\log \rho)'(h(p_j))(n+\norm{\nabla h(p_j)}^2)-n\Theta(p_j)H_1(p_j)
$$
we get
$$
-(\log\rho)'(h^*)+\liminf_{j\ra+\infty}H_1(p_j)\leq 0.
$$
Since
$$
\liminf_{j\ra+\infty}H_1(p_j)\geq \sqrt{H_2}>0,
$$
and $(\log\rho)'(h^*)\leq (\log\rho)'(h)$, it must be $(\log\rho)'(h)>0$, which means $\rho'(h)>0$ on $\Sigma$.
On the other hand, in the case where $\Theta>0$ with respect to this orientation, let us show that $\rho'(h)<0$. To do that, we apply the Omori-Yau maximum principle to the Laplacian to assure the existence of a sequence $\{q_j\}$ with the following properties
\begin{align*}
(i) \quad & \lim_{j\ra +\infty} h(q_j)=\inf_\Sigma h=h_*,\\
(ii) \quad & \norm{\nabla h(q_j)}<\frac{1}{j},\\
(iii) \quad & \Delta h(q_j)>-\frac{1}{j}.
\end{align*}
Therefore, making $j\ra+\infty$ in the following inequality
$$
-\frac1j<\Delta h(q_j)=-(\log \rho)'(h(q_j))(n+\norm{\nabla h(q_j)}^2)-n\Theta(q_j)H_1(q_j)
$$
and reasoning exactly as before we conclude that
$0>(\log\rho)'(h_*)\geq (\log\rho)'(h)$ and it must be $(\log\rho)'(h)<0$, which means $\rho'(h)<0$ on $\Sigma$.

Therefore, we have that for the chosen orientation
$$
(\log\rho)'(h)\Theta<0.
$$
By equation \eqref{lrsigmalor}
$$
L_1 \sigma(h)=-c_1\rho(h)((\log\rho)'(h)H_1+\Theta H_2).
$$
Consider the operator
$$
\mathcal{L}=-\frac{1}{\Theta}\frac{c_1}{c_0}(\log \rho)'(h)\Delta+L_1=\mathrm{Tr}(\mathcal{P}\circ\hess),
$$
where
\begin{displaymath}
\mathcal{P}=-(n-1)\frac{(\log \rho)'(h)}{\Theta}I+P_1=(n-1)\Big|\frac{(\log \rho)'(h)}{\Theta}\Big|I+P_1
\end{displaymath}
is positive definite. Since $|1/\Theta|\leq 1$, then $\sup_\Sigma|(\log\rho)'(h)|<+\infty$ and $\sup_\Sigma H_1<+\infty$,
$$
\mathrm{Tr}\mathcal{P}=c_1\Big(\Big|\frac{(\log \rho(h))'}{\Theta}\Big|+H_1\Big)< +\infty.
$$
Hence $\mathcal{L}$ is an elliptic operator and the trace of $\mathcal{P}$ is bounded from above.
By Corollary \ref{OYsigma} (see also Remark \ref{remarkmarco}) we can then apply the Omori-Yau maximum principle.
Since $h^*<+\infty$ there exists a sequence $\{p_j\} \subset \Sigma$ such that
\begin{align*}
\lim_{j \ra +\infty} (\sigma \circ h) (p_j)&=(\sigma \circ h)^*=\sigma(h^*),\\
\norm{\nabla(\sigma \circ h)(p_j)}&=\rho(h(p_j))\norm{\nabla h(p_j)}<\frac 1j,\\
\mathcal{L}(\sigma \circ h)(p_j)&<\frac 1j.
\end{align*}
Using
$$
\mathcal{L} \sigma(h)=-\frac{c_1}{\Theta}\rho(h)(-(\log\rho)'(h)^2+\Theta^2H_2),
$$
taking the limit for $j \ra +\infty$ and observing that $\Theta(p_j) \ra \sgn{\Theta}=\pm 1$ as $j \ra +\infty$, we find
$$
0 \geq \sgn{\Theta}((\log\rho)'(h^*)^2-H_{2}).
$$
On the other hand, since $h$ is bounded from below, we can find a sequence $\{q_j\} \subset \Sigma$ such that
\begin{align*}
\lim_{j \ra +\infty} (\sigma \circ h) (q_j)&=(\sigma \circ h)_*=\sigma(h_*),\\
\norm{\nabla (\sigma \circ h)(q_j)}&=\rho(h(q_j))\norm{\nabla h(q_j)}<\frac{1}{j},\\
\mathcal{L} (\sigma \circ h)(q_j) &> -\frac{1}{j}
\end{align*}
Hence, proceeding as above we find
$$
0 \leq \sgn{\Theta}((\log\rho)'(h_*)^2-H_{2}).
$$
Thus, taking into account that $(\log\rho)'(h)\Theta<0$ snd to the fact that $(\log\rho)'$ is a decreasing function, we get $h^*=h_*$.\\
For the general case $k\geq3$, since there exists an elliptic point and $H_k>0$, it follows by Lemma \ref{lemmaelliptlrlor} that $H_j >0$ and the operators $L_j$ are elliptic for all $1 \leq j \leq k-1$ with respect to an appropriate orientation. Reasoning as in the case $k=2$, one can see that $(\log\rho)'(h)\Theta<0$ for that orientation.
Furthermore, since, by the Newton inequalities
$$
H_j \leq H_1^j<+\infty,
$$
each $H_j$ is bounded from above.
By equation \eqref{lrsigmalor}
$$
L_{k-1} \sigma(h)=-c_{k-1}\rho(h)((\log\rho)'(h)H_{k-1}+\Theta H_k).
$$
Consider the operator
\begin{align*}
\mathcal{L}=&\mathrm{Tr}\Big(\Big[\sum_{i=0}^{k-1}\frac{c_{k-1}}{c_i}\Big(-\frac{(\log \rho)'(h)}{\Theta}Big)^{k-1-i}P_i\Big]\circ \hess \Big)\\
=&\sum_{i=0}^{k-1}\frac{c_{k-1}}{c_i}\Big(-\frac{(\log \rho)'(h)}{\Theta}\Big)^{k-1-i}L_i\\
=&\sum_{i=0}^{k-1}\frac{c_{k-1}}{c_i}\Big|\frac{(\log \rho)'(h)}{\Theta}\Big|^{k-1-i}L_i.
\end{align*}
Since $\mathcal{L}$ is a positive linear combination of the $L_i$'s, it is elliptic. Moreover
\begin{align*}
\mathrm{Tr}(\mathcal{P})=&c_{k-1}\sum_{i=0}^{k-1}\Big|\frac{(\log\rho)'(h)}{\Theta}\Big|^{k-1-i}H_i\\
\leq &c_{k-1}\sum_{i=0}^{k-1}|(\log\rho)'(h)|^{k-1-i}H_1^i\\
\leq&c_{k-1}\sum_{i=0}^{k-1}\sup_\Sigma|(\log\rho)'(h)|^{k-1-i}\sup_\Sigma H_1^i
\end{align*}
is bounded from above.
Similarly as in the proof of (\ref{mathcallrcompact}), it is easy to prove by induction on $k$ that
\begin{equation}\label{mathcallr}
\mathcal{L} \sigma(h)=\frac{c_{k-1}}{(-\Theta)^{k-1}}\rho(h)(-((\log\rho)'(h))^{k}+(-1)^k\Theta^kH_k).
\end{equation}
We can then apply the Omori-Yau maximum principle to the operator $\mathcal{L}$.
Since $h^*<+\infty$ there exists a sequence $\{p_j\} \subset \Sigma$ such that
\begin{align*}
\lim_{j \ra +\infty} (\sigma \circ h) (p_j)&=(\sigma \circ h)^*=\sigma(h^*),\\
\norm{\nabla(\sigma \circ h)(p_j)}&=\rho(h(p_j))\norm{\nabla h(p_j)}<\frac 1i,\\
\mathcal{L}(\sigma \circ h)(p_j)&<\frac 1j.
\end{align*}
Hence, taking the limit in \eqref{mathcallr} for $j \ra +\infty$ and observing that $\Theta \ra \sgn{\Theta}=\pm 1$ as $j \ra +\infty$, we find
$$
0 \geq \sgn{\Theta}(((\log\rho)'(h^*))^k-H_k).
$$
On the other hand, since $h$ is bounded from below, we can find a sequence $\{q_j\} \subset \Sigma$ such that
\begin{align*}
\lim_{j \ra +\infty} (\sigma \circ h) (q_j)&=(\sigma \circ h)_*=\sigma(h_*),\\
\norm{\nabla (\sigma \circ h)(q_j)}&=\rho(h(q_j))\norm{\nabla h(q_j)}<\frac{1}{j},\\
\mathcal{L} (\sigma \circ h)(q_j) &> -\frac{1}{j}
\end{align*}
Hence, proceeding as above we find
$$
0 \leq\sgn{\Theta}(((\log\rho)'(h_*))^k-H_k).
$$
Thus we conclude as in case $k=2$.
\end{proof}
\section{Further results for complete spacelike hypersurfaces}\label{further}
Recall that a spacetime obeys the null convergence condition (NCC) if its Ricci curvature is nonnegative on lightlike directions. In the case of a generalized Robertson-Walker spacetime $-I\times_\rho\p^n$ this is equivalent to
\begin{equation}\label{NCC}
\ricc_\p\geq(n-1)\sup_I(\rho^2(\log\rho)'')\pair{,}_\p.
\end{equation}
In \cite[Theorem 6]{montiel2} (see also Theorem 9.1 in \cite{aliascolares}) it was proved that the only compact spacelike hypersurfaces with constant mean curvature in a spatially closed generalized Robertson-Walker spacetime obeying the NCC are the slices, unless in the case where the ambient space is isometric to the de Sitter spacetime in a neghbourhood of $\Sigma$, which must be a round umbilical hypersphere. Moreover, the latter case cannot occur if we assume that the inequality in \eqref{NCC} is strict. Our first result in this section extends this to the complete noncompact case as follows
\begin{theorem}\label{thNCC1}
Let $-I \times_{\rho} \p^n$ be a generalized Robertson-Walker spacetime obeying the strict null convergence condition, that is, satisfying
\begin{equation}\label{strictNCC}
\ricc_\p > (n-1)\sup_I ((\log\rho)'' \rho^2)\pair{,}_\p.
\end{equation}
Let $f:\Sigma^n \ra -I \times_{\rho} \p^n$ be a complete spacelike hypersurface of constant mean curvature contained in a slab $\Omega(t_1,t_2)$.
Suppose that $\Sigma^n$ is parabolic and $\sup_\Sigma|\Theta|<+\infty$. Then $f(\Sigma^n)$ is a slice.
\end{theorem}
For the proof of the theorem we begin recalling the following computational result from \cite{aliascolares}
\begin{lemma}[Corollary 8.2 in \cite{aliascolares}]\label{lapltheta}
Let $\Sigma^n$ be a spacelike hypersurface immersed into a generalized Robertson-Walker spacetime $-I\times_{\rho}\p^n$, with angle function $\Theta$ and height function $h$. Let $\hat{\Theta}=\rho(h)\Theta$. Then we have
\begin{align*}
\Delta \hat{\Theta}=&n\rho(h)\pair{\nabla h, \nabla H_1}+n\rho'(h)H_1+n\hat{\Theta}(nH_1^2-(n-1)H_2)\\
&+\hat{\Theta}(\ricc_{\p}(N^*,N^*)-(n-1)(\log\rho)''(h)\norm{\nabla h}^2).
\end{align*}
\end{lemma}
\begin{proof}[Proof of Theorem \ref{thNCC1}]
Let us choose on $\Sigma$ the orientation such that $\Theta<0$ and consider the function $\phi=H_1\sigma(h)+\hat{\Theta}$. Since the mean curvature is constant, by Equation \eqref{lrsigmalor} and Lemma \ref{lapltheta} we have
$$
\Delta\phi=\hat{\Theta}(n(n-1)(H_1^2-H_2)+\ricc_{\p}(N^*,N^*)-(n-1)(\log\rho)''(h)\norm{\nabla h}^2).
$$
Reasoning as in the proof of Theorem 9.1 in \cite{aliascolares}, it follows from the hypotheses that $\Delta\phi\leq0$ on $\Sigma$. Since $\sup_\Sigma|\Theta|<+\infty$ and $\Sigma$ is contained in a slab, $\phi$ is bounded from below. Moreover, $\Sigma$ being parabolic implies that $\phi$ must be constant and $\Delta\phi=0$.
In particular
$$
\ricc_{\p}(N^*,N^*)-(n-1)(\log\rho)''(h)\norm{\nabla h}^2=0.
$$
Observe that
$$\norm{\nabla h}^2=\norm{N^*}^2=\rho^2(h)\pair{N^*,N^*}_\p.
$$
Therefore by the strict null convergence condition, it must be $\nabla h=0$ and hence $\Sigma$ is a slice.
\end{proof}
In order to extend this reasoning to the higher order mean curvature, we need the following computational result.
\begin{lemma}[Corollary 8.4 in \cite{aliascolares}]\label{lrthetasecconst}
Let $\Sigma^n$ be a spacelike hypersurface immersed into a RW spacetime $-I\times_{\rho}\p^n$, with angle function $\Theta$ and height function $h$. Assume that $\p^n$ has constant sectional curvature $\kappa$ and let $\hat{\Theta}=\rho(h) \Theta$. Then, for every $k=0,...,n-1$ we have
\begin{align*}
L_k \hat{\Theta}=&{n \choose {k+1}}\rho(h)\pair{\nabla h, \nabla H_{k+1}}+\rho'(h)c_k H_{k+1}\\
&+\hat{\Theta}\Big(\frac{\kappa}{\rho^2(h)}-(\log \rho)''(h)\Big)(\norm{\nabla h}^2c_kH_k-\pair{P_k \nabla h,\nabla h})\\&+\hat{\Theta}{n \choose{k+1}} (nH_1H_{k+1}-(n-k-1)H_{k+2})
\end{align*}
\end{lemma}

For the general case, we replace the Laplacian operator by the following operator
$$
\mathfrak{L} f=\mathrm{div}(P_{k-1} \nabla f),
$$
where $f \in C^\infty(\Sigma)$.
Using Lemma 3.1 in \cite{aliasbrasilcolares} we find that
\begin{align*}
\mathfrak{L} f=&\pair{\mathrm{div}P_{k-1},\nabla f}+L_{k-1}f\\
=&\sum_{j=0}^{k-2}\sum_{i=1}^n(-1)^{k-2-j}\pair{\overline{\R}(E_i,A^{k-2-j}\nabla f)N,P_j E_i}+L_{k-1} f\end{align*}
It follows by Equation 6.16 in \cite{aliascolares} that, in the case where $\p^n$ has constant sectional curvature $\kappa$, the operator $\mathfrak{L}$ becomes
\begin{equation}\label{frakL}
\mathfrak{L}f=(n-k+1)\Theta\Big(\frac{\kappa}{\rho^2(h)}-(\log \rho)''(h)\Big)\pair{P_{k-2}\nabla h,\nabla f}+L_{k-1}f.
\end{equation}
We introduce the next
\begin{definition}
We will say that the hypersurface $\Sigma^n\hookrightarrow -I \times_{\rho} \p^n$ is $\mathfrak{L}$-parabolic if the only bounded above $C^1$ solutions of the differential inequality
$$
\mathfrak{L} f \geq 0
$$
are constant.
\end{definition}
The following result is a special case of Theorem 2.6 in \cite{pirise2}
\begin{theorem}
Let $\Sigma^n\hookrightarrow -I \times_{\rho} \p^n$ be a complete spacelike hypersurface. If
\begin{equation}\label{eqrpar}
\Big(\sup_{\partial B_t} H_{k-1} \mathrm{vol}(\partial B_t)\Big)^{-1}\notin L^1(+\infty),
\end{equation}
where $\partial B_t$ is a geodesic sphere of radius $t$, then $\Sigma^n$ is $(k-1)$-parabolic.
\end{theorem}
Now we are ready to establish the second main result of this section, which extends Theorem 9.2 in \cite{aliascolares} to the complete case, at least when $\p^n$ has constant sectional curvature.
\begin{theorem}\label{main2complete}
Let $-I \times_{\rho} \p^n$ be a Robertson-Walker space and denote the constant sectional curvature of $\p^n$ by $\kappa$. Let $f:\Sigma^n \ra -I \times_{\rho} \p^n$ be a complete spacelike hypersurface of constant $k$-mean curvature, $k\geq 2$, contained in a slab $\Omega(t_1,t_2)$ on which $\rho'$ does not change sign and
\begin{equation}\label{seccurv}
\kappa > \max_{[t_1,t_2]} ((\log\rho)'' \rho^2).
\end{equation}
Suppose that $\Sigma^n$ satisfies condition \eqref{eqrpar} and either
\begin{itemize}
\item[(i)] $k=2$ and $H_2>0$ or
\item[(ii)] $k\geq3$ and  there exists an elliptic point $p \in \Sigma^n$.
\end{itemize}
If $\sup_\Sigma|H_1|<+\infty$ and $\sup_\Sigma|\Theta|<+\infty$, then $f(\Sigma^n)$ is a slice.
\end{theorem}
\begin{proof}
Assume that $\rho'(h)\geq0$ and let us see that, for an appropriate orientation of $\Sigma$, one has $H_1>0$ and $\Theta<0$. Consider first the case $k=2$. Since $H_2 >0$ it follows by Lemma \ref{lemmaelliptl1lor} that the operator $L_1$ is elliptic with respect to the orientation for which $H_1>0$ (see the proof of Lemma 3.2 in \cite{aliascolares} for the details). Let us see that this orientation gives $\Theta<0$. To do that, we apply the Omori-Yau maximum principle to the Laplacian, which holds on $\Sigma$ by Corollary \ref{OYsigma} and Remark \ref{remarkmarco}.
This assures the existence of a sequence $\{q_j\}$ with the following properties
\begin{align*}
(i) \quad & \lim_{j\ra +\infty} h(q_j)=\inf_\Sigma h=h_*,\\
(ii) \quad & \norm{\nabla h(q_j)}<\frac{1}{j},\\
(iii) \quad & \Delta h(q_j)>-\frac{1}{j}.
\end{align*}
Therefore, making $j\ra+\infty$ in the following inequality
$$
-\frac1j<\Delta h(q_j)=-(\log \rho)'(h(q_j))(n+\norm{\nabla h(q_j)}^2)-n\Theta(q_j)H_1(q_j)
$$
we get
$$
(\log\rho)'(h_*)+\sgn(\Theta)\liminf_{j\ra+\infty}H_1(q_j)\leq 0.
$$
Since
$$
\liminf_{j\ra+\infty}H_1(q_j)\geq \sqrt{H_2}>0,
$$
and $(\log\rho)'(h_*)\geq 0$, it must be $\sgn(\Theta)=-1$, which means $\Theta<0$ on $\Sigma$.\\
For the general case $k\geq3$, since there exists an elliptic point and $H_k>0$, it follows by Lemma \ref{lemmaelliptlrlor} that $H_j >0$ and the operators $L_j$ are elliptic for all $1 \leq j \leq k-1$ with respect to an appropriate orientation. In particular, $H_2>0$ and reasoning as in the case $k=2$, one can see that $\Theta<0$ for that orientation.\\
Next we consider the function
$$
\phi=H_k^{\frac{1}{k}}\sigma(h)+\hat{\Theta},
$$
with $\hat{\Theta}=\rho(h)\Theta$.
Since $\p^n$ has constant sectional curvature $\kappa$, it follows by Equation \eqref{frakL} that
\begin{align*}
\mathfrak{L} \phi=&(n-k+1)\Theta\Big(\frac{\kappa}{\rho^2(h)}-(\log \rho)''(h)\Big)\pair{P_{k-2}\nabla h,\nabla \phi}+L_{k-1} \phi \\
=&(n-k+1)\hat{\Theta}\Big(\frac{\kappa}{\rho^2(h)}-(\log \rho)''(h)\Big)\pair{P_{k-2}\nabla h,\nabla h}\\&+(n-k+1)\hat{\Theta}\Big(\frac{\kappa}{\rho^2(h)}-(\log \rho)''(h)\Big)\pair{P_{k-2}A\nabla h,\nabla h}\\&+H_k^{\frac1k}L_{k-1} \sigma(h)+L_{k-1} \hat{\Theta}.
\end{align*}
Using Equation \eqref{lrsigmalor} and Corollary \ref{lrthetasecconst} we find
\begin{align}\label{Lrphi}
\mathfrak{L} \phi=&-c_{k-1}\rho'(h)H_k^{\frac1k}(H_{k-1}-H_k^{\frac{k-1}{k}})\\ \nonumber
&+(n-k+1)\hat{\Theta}\Big(\frac{\kappa}{\rho^2(h)}-(\log \rho)''(h)\Big)H_k^{\frac1k}\pair{P_{k-2}\nabla h,\nabla h}\\ \nonumber
&+(n-k)\hat{\Theta}\Big(\frac{\kappa}{\rho^2(h)}-(\log \rho)''(h)\Big)\pair{P_{k-1} \nabla h,\nabla h}\\ \nonumber
&+\hat{\Theta}{n \choose k}(nH_1H_k-(n-k)H_{k+1}-kH_k^{\frac{k+1}{k}}).\nonumber
\end{align}
Using Garding inequalities it is easy to prove that the first and the last terms are nonnegative. By the fact that each $P_j$ is an elliptic operator, $j=0,...,k-1$, and by equation \eqref{seccurv} it follows that also all the remaining terms in the previous equation are nonnegative. Hence $\mathfrak{L} \phi \leq 0$. Since $\sup_\Sigma|\Theta|<+\infty$ and the hypersurface is contained in a slab, then $\phi$ is bounded from below. Moreover, by assumption \eqref{eqrpar}, $\Sigma^n$ is $\mathfrak{L}$-parabolic. Therefore we conclude that $\phi$ has to be constant. In particular, $\mathfrak{L} \phi=0$. Hence each term of
Equation \eqref{Lrphi} must vanish.
In particular equality
$$nH_1H_k-(n-k)H_{k+1}-kH_k^{\frac{k+1}{k}}=0$$
implies that $\Sigma$ is a totally umbilical hypersurface. Moreover, since each $P_j$, $j=0,...,k-1$ is an elliptic operator and since \eqref{seccurv} holds, we conclude that $\nabla h=0$ and hence $f(\Sigma)$ is a slice.
\end{proof}

\bigskip

\bibliographystyle{amsplain}
\bibliography{biblioAIRlor}

\end{document}